\documentclass{amsart}
\usepackage[utf8]{inputenc}
\usepackage{amsmath}
\usepackage{changepage}
\usepackage[backend=bibtex, sorting=none, bibstyle=alphabetic, citestyle=alphabetic, sorting=nyt,maxbibnames=99]{biblatex}
\usepackage{graphicx}
\usepackage{caption}

\newtheorem{theorem}{Theorem}[section]
\theoremstyle{remark}
\newtheorem*{remark}{Remark}

\bibliography{main.bib}

\begin{document}

\title{New formulas for the Riemann Zeta Function}

\author[A. Akula \and G. Hiary]{Aditya Akula \and Ghaith Hiary}

\address{
    AA: Georgia Institute of Technology, 336705 Georgia Tech Station, Atlanta, GA 30332, USA.
}
\email{akula.aditya@gmail.com}
\address{
    GH: Department of Mathematics, The Ohio State University, 231 West 18th
    Ave, Columbus, OH 43210, USA
}
\email{hiary.1@osu.edu}
\subjclass[2010]{Primary: 11Y35, 11M06.}
\keywords{The Riemann zeta function, Dirichlet series}

\maketitle

\begin{abstract}
    A new method
    for continuing the usual Dirichlet series that defines the Riemann
    zeta function $\zeta(s)$ is presented. Numerical
    experiments demonstrating the computational efficacy of the resulting
    continuation are discussed.
\end{abstract}

\section{Introduction}
The usual Dirichlet series defining the Riemann Zeta Function $\zeta(s)$ is   
\begin{equation}\label{zeta_dirichlet}
\zeta(s) = \sum_{n=1}^{\infty} n^{-s},
\end{equation}
which converges in the half-plane $\operatorname{Re}(s) >1$ only. We construct a
new family of
linear combinations of subsums of \eqref{zeta_dirichlet} 
along arithmetic progressions to achieve convergence in arbitrarily large half-planes.

To this end, let $m$ be a positive integer, and let $d(m)$ denote the number of
positive integer divisors of $m$. Our method enables continuing $\zeta(s)$ to
the larger half-plane $\operatorname{Re} (s)>2-d(m)$, provided $d(m) \geq 4$. 

An essential ingredient of our method is certain Dirichlet series weights $b_k$.
These weights are $m$-periodic, so $b_{k+m}=b_k$ for all $k\ge 1$, and may be
found on demand by a routine calculation of nonzero vectors in the kernel of a certain $d(m)\times d(m)$ singular matrix $A$. 
Higher values of $m$ with more divisors $d(m)$ allow for more complicated combinations and a larger half-plane of convergence. 

The formulas we present may be reminiscent of the well-known formula \cite[p. 16]{Titchmarsh},
\begin{equation}\label{eta}
\zeta(s)\left(1-\frac{2}{2^s} \right)=
\sum_{n=0}^{\infty} {\left(\frac{1}{(2n+1)^s} - \frac{1}{(2n+2)^s}\right)},
\end{equation}
which is convergent for $\operatorname{Re}(s)>0$. Our original motivation was,
in fact, to generalize \eqref{eta} by considering 
analogies with numerical differentiation formulas. 

Let $D$ denote the set of positive divisors of $m$, and label the elements of
$D$ in increasing order, so $d_1 = 1$ and $d_{d(m)} = m$. These will be used to
separate the integers from $1$ to $m$ into subsets according to $\gcd(k,m)$ for
$1 \leq k\leq m$. Specifically, in section \ref{derivation}, we will construct
formulas of the form 
\begin{equation}\label{mainformula}
    \zeta(s) \cdot \left( \sum_{j=1}^{d(m)} \frac{a_j}{(d_j)^s} \right) =
    \sum_{n=0}^{\infty} \sum_{k=1}^{m} \frac{b_k}{(mn+k)^s},
\end{equation}
such that the right-side is convergent in the half-plane $\operatorname{Re}(s) >
2- d(m)$. Note that the series on the right-side of \eqref{mainformula} is related to
$\zeta(s)$ in a simple way, as one may divide by the finite sum over $j$ on the
left-side (provided that sum is nonzero) to arrive at zeta. 

In section \ref{numerical_experiments}, 
we also study the error resulting from truncating 
the sum over $n$ in \eqref{mainformula} at some $n=N$. 
An interesting finding concerns the convergence behavior on the
critical line. 
Figure~\ref{conv_rate} displays a typical example, illustrating
plateau-decay behavior, initially, and dramatic improvements in accuracy for some choices of the truncation
point $N$.

\begin{figure}[h]
    \centering
    \includegraphics[scale=0.45]{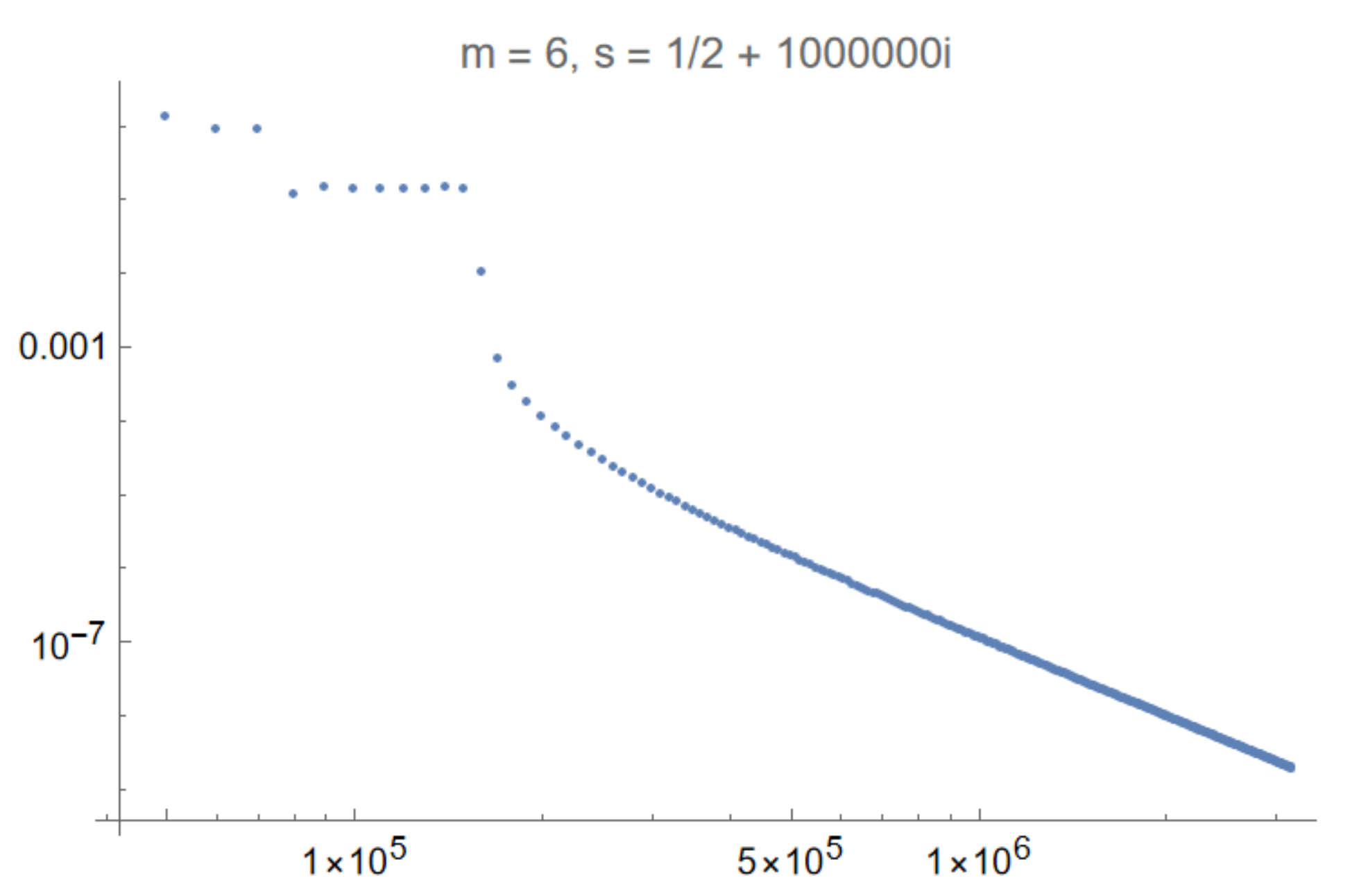}
    \includegraphics[scale=0.45]{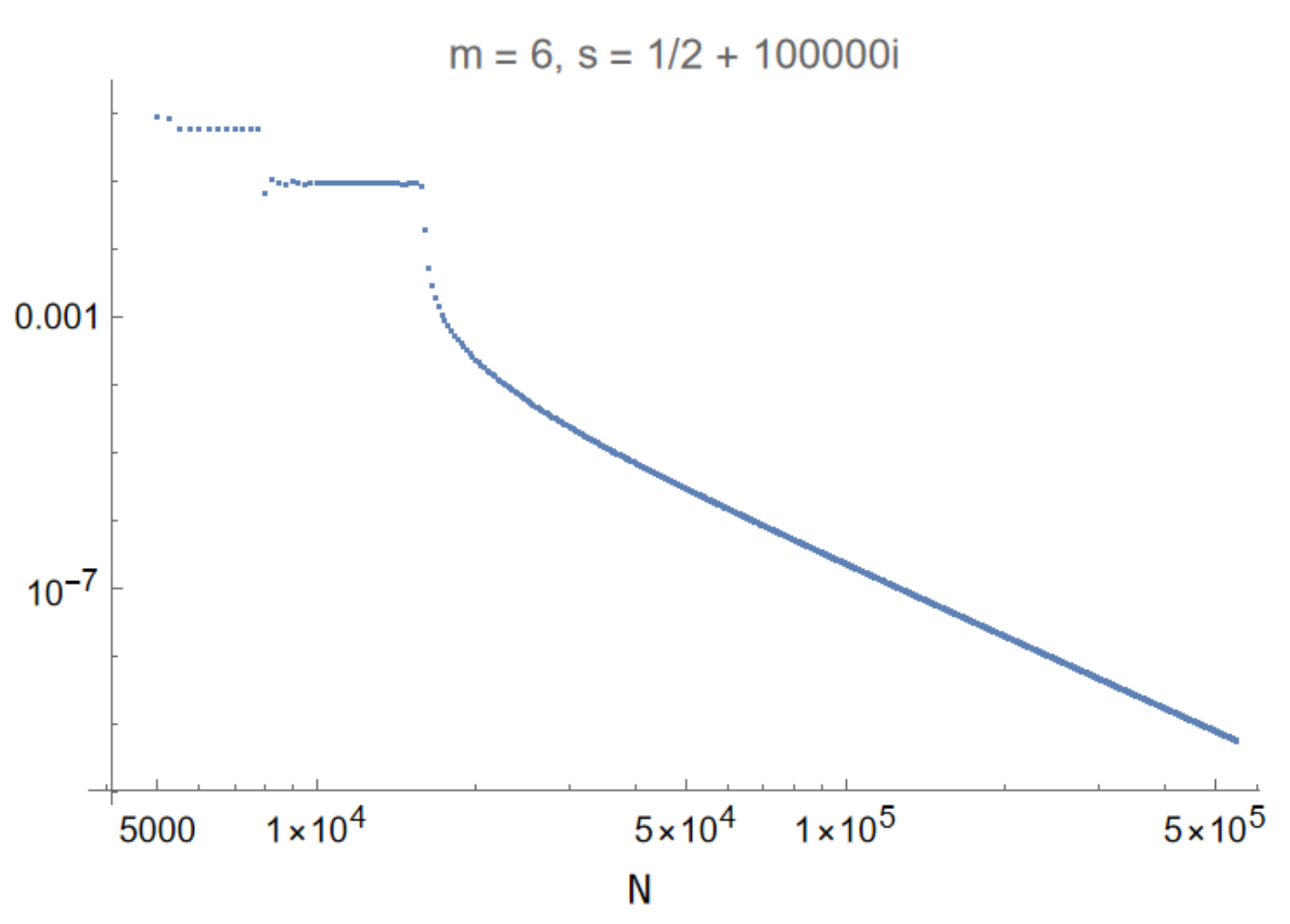}
    \caption{\small $\log$-$\log$ plots of the truncation error
    against the number of terms $N$, where we used the formula
    \eqref{mainformula}
     corresponding to the vector 
    $\boldsymbol{a}=[1,-5,5,-1]^T$. The range of $N$ 
    in the first graph is from $5\cdot 10^4$ to $3.15\cdot 10^6$, going in
    increments of $1000$. The range of $N$ in the second graph 
    is from $5\cdot 10^3$ to $5.5\cdot 10^6$, going in
    increments of $250$.} 
    \label{conv_rate}
\end{figure}

\section{Derivation}\label{derivation}
\subsection{Overview}
Our goal is to determine coefficients $a_1,\ldots,a_{d(m)}$ and $b_1,\ldots,
b_m$ such that formula \eqref{mainformula} holds and such that the right-side in
the formula converges in the half-plane $\Re(s) > 2 - d(m)$. 

Multiplying the Dirichlet series in \eqref{zeta_dirichlet} term-by-term by $1/k^s$ ``filters'' the terms in the series to only those divisible by $k$. We will use this, multiplying $\zeta(s)$ by $a_j/d_j^s$ for some set of coefficients $a_j$, where the $d_j$ are in $D$. As will be explained later, the coefficients $a_1,\ldots, a_{d(m)}$ will in turn be used to determine the coefficients $b_1,\ldots, b_m$. 
This will result in a formula of the form \eqref{mainformula} with the
right-side convergent for $\operatorname{Re}(s)>2-d(m)$, which is 
far beyond the original half-plane of convergence. 

Let first motivate the conditions we intend to impose on the $b_k$'s.
Consider the outcome on 
substituting non-positive integer values of $s$ into the inner sum in \eqref{mainformula}. In
this case, the inner sum simplifies to a polynomial in $n$. Summing the values
of this polynomial over $n$ gives a divergent sum unless the polynomial in
question is identically zero. For instance, if $m=6$, then $d(m)=4$ and
the inner sum in formula \eqref{mainformula} becomes

\begin{equation}\label{m6example}
b_1 (6n+1)^{-s} + b_2 (6n+2)^{-s} + \cdots+ b_6 (6n+6)^{-s}.   
\end{equation}
\newline
We would like this to be identically zero 
when $s=0,-1,-2, -3$, if possible. It is enlightening to consider what conditions on
the $b_k$'s arise as we progressively impose these requirements.
When $s=0$ and $n\ge 0$, the expression \eqref{m6example} turns into 
$$b_1 + b_2+\cdots + b_6.$$
To ensure that this expression is identically zero, we need $b_1 + b_2 +\cdots+ b_6 = 0$. When $s=-1$, the expression \eqref{m6example} turns into the following polynomial in $n$.
$$6(b_1 + b_2+\cdots + b_6)n + (b_1 + 2b_2 +\cdots+ 6b_6).$$ 
So, in view of our earlier requirement that $b_1 + b_2 +\cdots+ b_6 = 0$, to
ensure that this last polynomial is identically zero we just need $b_1 + 2b_2 +\cdots+ 6b_6=0$. Next, when $s=-2$, the sum \eqref{m6example} turns into
$$36(b_1 + b_2+\cdots + b_6)n^2 + 12(b_1 + 2b_2 +\cdots+ 6b_6)n+(b_1+2^2b_2+\cdots+6^2b_6).$$ 
So to ensure that this new polynomial in $n$ is identically zero, we just need
$b_1 + 2^2b_2 +\cdots+ 6^2 b_6=0$. Lastly, when $s=-3$, we obtain one additional condition that $b_1 + 2^3b_2 +\cdots+ 6^3 b_6=0$. 

In summary, we obtain a system of equations in $b_1,\ldots,b_6$. This system can
be represented by a matrix $B$ of dimension $d(m)\times m = 4\times 6$. (The
matrix $B$ is different from the matrix $A$ mentioned in the introduction.)
Hence, we just need to find nonzero vectors in the kernel of $B$. We find that a
basis for the kernel of $B$ is given by
$$b_1=4,\quad b_2=-15,\quad b_3=20,\quad b_4=-10,\quad b_5=0,\quad b_6=1$$ 
and
$$b_1=1,\quad b_2=-4,\quad b_3=6,\quad b_4=-4,\quad b_5=1,\quad b_6=0.$$ 
Either of the vectors determined by these values of the $b_k$'s will accomplish
the desired goal, as will any other nonzero vector in the kernel of $B$. 

In general, as will transpire in the next subsection, for the double-sum in formula \eqref{mainformula} to be
convergent in the larger half-plane $\operatorname{Re}(s)>2-d(m)$,  it is enough
to ensure that
\begin{equation}\label{bkequation}
    \sum_{k=1}^{m} b_k \cdot k^{-s}=0\qquad \qquad s=0,-1,\ldots,1-d(m).
\end{equation}
Briefly, this is because for $s=0,-1,\ldots,1-d(m)$, 
if we expand $b_k (mn  + k)^{-s}$ as a polynomial in $n$ using
the binomial theorem and sum across all $k$, then we can group the resulting terms by
power of $n$. The coefficients of these powers of $n$ are of the form
\eqref{bkequation}. Thus, in seeking convergence in the half-plane
$\operatorname{Re}(s)>2-d(m)$, we will 
want the coefficient of each power of $n$ to be $0$. 

However, even under these conditions, we cannot yet conclude that the right-side
of \eqref{mainformula} converges in the half-plane
$\operatorname{Re}(s)>2-d(m)$. This is because each inner sum on the right-side
of \eqref{mainformula} must be taken in full and cannot be truncated. So we
cannot directly apply the well-known theorem that if a Dirichlet series
converges at $s = x_0 + iy_0$, then it converges in the half plane
$\operatorname{Re} (s)>x_0$ and is analytic in that half-plane; see
\cite[Theorem 11.12]{Apostol} for an example statement of the said theorem. 

\subsection{The $b_k$ coefficients and convergence}\label{convval}
Let us denote the right-side of \eqref{mainformula} by $Z_m(s)$.
We prove the following.

\begin{theorem}\label{maintheorem}
If the coefficients $b_1,\ldots,b_m$ satisfy the conditions 
    in \eqref{bkequation}, then $Z_m(s)$ converges in the half-plane
    $\operatorname{Re}(s)>2-d(m)$ and is analytic there. 
\end{theorem}

\begin{proof}
Suppose initially that $\operatorname{Re}(s)>1$, so $Z_m(s)$ is absolutely
convergent. We use the Taylor expansion to re-express the inner sums in
$Z_m(s)$. We may restrict our analysis to inner sums with $n\ge n_0 \ge 3$, say.
 This restriction does not pose a problem for analysis of
convergence since $n=0,1,\ldots,n_0-1$ correspond to a finite subsum. 
So, let us write
\begin{equation}
    \begin{split}
\frac{1}{(mn+k)^s}&=\frac{1}{(mn+m/2+(k-m/2))^s}\\
&= \frac{1}{(mn+m/2)^s}\cdot \left.\frac{1}{(1+z)^s}\right|_{z=\frac{k-m/2}{mn+m/2}}.
    \end{split}
\end{equation}
We expand $(1+z)^{-s}$ in powers of $z$.
\begin{equation}\label{z_expansion}
    (1+z)^{-s} = \sum_{\ell=0}^{\infty} f_{\ell}(s)z^{\ell}.
\end{equation}
So, $f_0(s)=1$, $f_1(s)=-s$, $f_2(s) = s(s+1)/2$, and in general 
\begin{equation}\label{fl_formula}
f_{\ell}(s) = \frac{(-1)^{\ell}}{\ell!} \prod_{u=0}^{\ell-1}{(s+u)}.
\end{equation}

    The coefficients $f_{\ell}(s)$ grow at most like a polynomial in $\ell$. 
More explicitly, we have $f_{\ell}(s)\ll \ell^{|s|+1}$, as can be
seen with the aid of basic properties of the $\Gamma$-function; see \cite[p.
    73]{davenport-book}, for example.
(If $s$ is a nonpositive integer, then the $f_{\ell}(s)$ are eventually
all zero.)
    Therefore, if $|z|<1$, 
then the exponential decay due to $|z|^{\ell}$ will dominate the polynomial growth
    in $f_{\ell}(s)$ in the expansion \eqref{z_expansion}. Consequently, if $|z|<1$, then the expansion \eqref{z_expansion} converges absolutely for any
value of $s$.

    Since $n\ge n_0\ge 3$ and $1\le k\le m$,
and considering that we plan to take 
$z=(k-m/2)/(mn+m/2)$, we see that the condition $|z|<1$ is satisfied in our
case. Therefore, we obtain 
\begin{equation}\label{taylorexpansion}
\sum_{n=n_0}^{\infty} \sum_{k=1}^{m} \frac{b_k}{(mn+k)^s}=
\sum_{n=n_0}^{\infty} \sum_{k=1}^m b_k  \sum_{\ell=0}^{\infty}
    \frac{f_{\ell}(s)}{(mn + m/2)^{s+\ell}} (k-m/2)^{\ell},
\end{equation}
where the sum over $\ell$ converges absolutely for any $s$.

    Applying the binomial theorem to the $(k-m/2)^{\ell}$ term in
    \eqref{taylorexpansion}, and then 
interchanging the order of summation in the 
    absolutely convergent double-sum over $\ell$ and $k$, 
and finally grouping the resulting terms by degree, gives that
the right-side in \eqref{taylorexpansion} is equal to
\begin{equation} \label{taylorexpansion2}
    \sum_{n=n_0}^{\infty} \sum_{\ell=0}^{\infty} \frac{f_{\ell}(s)}{(mn + m/2)^{s+\ell}}\sum_{r=0}^{\ell} \binom{\ell}{r} (-m/2)^{\ell-r} \sum_{k=1}^m b_k k^r.
\end{equation}
We now appeal to the conditions \eqref{bkequation}; namely, that for each integer $r$ satisfying $0 \leq r \leq d(m) - 1$,
\begin{equation}\label{bkvanishing}
\sum_{k=1}^{m} b_k k^{r}=0. 
\end{equation}
Using this, we see that 
    the sum over $r\in [0,\ell]$ in \eqref{taylorexpansion2} vanishes if $\ell < d(m)$.
Therefore, the expression in \eqref{taylorexpansion2} is equal to
\begin{equation} \label{taylorexpansion3}
    \sum_{n=n_0}^{\infty} \sum_{\ell=d(m)}^{\infty} \frac{f_{\ell}(s)}{(mn + m/2)^{s+\ell}}\sum_{r=d(m)}^{\ell} \binom{\ell}{r} (-m/2)^{\ell-r} \sum_{k=1}^m b_k k^r.
\end{equation}

We claim the the sum \eqref{taylorexpansion3} is absolultey convergent for 
any $s$ in the half-plane $\operatorname{Re}(s) + d(m) > 2$, 
and hence provides the desired continuation of $Z_m(s)$.
For by the geometric-arithmetic mean inequality, 
and since $n\ge n_0\ge 3$,
\begin{equation}
    mn+\frac{m}{2} = m(n-n_0+1)+\left(n_0-\frac{1}{2}\right)m \ge
    2m\sqrt{n-n_0+1}.
\end{equation}
Thus, for $\ell\ge d(m)$ and $\operatorname{Re}(s)+d(m) >0$,
\begin{equation}
    \left(mn+\frac{m}{2}\right)^{\operatorname{Re}(s)+\ell} \ge 
    (n-n_0+1)^{\frac{\operatorname{Re}(s)+d(m)}{2}}
    (2m)^{\operatorname{Re}(s)+\ell}. 
\end{equation}
Hence, by the triangle inequality, 
the sum in \eqref{taylorexpansion3} is bounded in size by
\begin{equation}
    \begin{split}
        & \zeta\left(\frac{\operatorname{Re}(s)+d(m)}{2}\right) 
        \sum_{\ell=d(m)}^{\infty}
        \frac{|f_{\ell}(s)|}{(2m)^{\operatorname{Re}(s)+\ell}}\sum_{r=d(m)}^{\ell}
    \binom{\ell}{r} (m/2)^{\ell-r} \sum_{k=1}^m |b_k| k^r \\
        & \le \zeta\left(\frac{\operatorname{Re}(s)+d(m)}{2}\right) 
         \sum_{k=1}^m |b_k| \sum_{\ell=d(m)}^{\infty}
        |f_{\ell}(s)|
        \frac{(k+m/2)^{\ell}}{(2m)^{\operatorname{Re}(s)+\ell}}\\
        & \le \zeta\left(\frac{\operatorname{Re}(s)+d(m)}{2}\right) 
         \left(\sum_{k=1}^m |b_k|\right) \sum_{\ell=d(m)}^{\infty}
        |f_{\ell}(s)| \left(\frac{3}{4}\right)^{\ell},
    \end{split}
\end{equation}
and, as pointed out earlier, the sum over $\ell$ converges absolutely for any
$s$. Hence, as claimed, 
the sum in \eqref{taylorexpansion3} converges absolutely
in the half-plane $\operatorname{Re}(s) + d(m) > 2$, and 
is therefore analytic in that half-plane.
\end{proof}

\begin{remark}
In view of \eqref{bkequation}, 
the first nonzero term in the sum over $\ell$ in \eqref{taylorexpansion3} is 
\begin{equation}\label{nonzeroterm}
    \frac{f_{d(m)}(s)}{(mn +m/2)^{s+d(m)}}\sum_{k=1}^{m} b_k k^{d(m)}.
\end{equation}
All subsequent terms have exponents with larger real part than
    $\operatorname{Re}(s)+d(m)$ in the denominator. 
    So it is possible that convergence occurs in the larger half-plane
    $\operatorname{Re}(s) > 1-d(m)$. Numerical experiments that we carried
    out seem consistent with this. 
\end{remark}

\subsection{Solving for the $a_j$ coefficients}
We now consider the left-side of the formula \eqref{mainformula}. For $\operatorname{Re} (s)>1$, we have
\begin{equation}
    \begin{split}
    \zeta(s) \cdot \left( \sum_{j=1}^{d(m)} \frac{a_j}{(d_j)^s} \right) & = \sum_{j=1}^{d(m)}\left(\frac{a_j}{(d_j)^s} + \frac{a_j}{(2d_j)^s} +\frac{a_j}{(3d_j)^s}\cdots \right) \\ 
    & = \frac{a_1}{1^s} + \frac{\sum_{d_j \, | 2} a_j}{2^s} + \frac{\sum_{d_j \,
        | 3} a_j}{3^s}\cdots, 
    \end{split}
\end{equation}
where we used absolute convergence to rearrange the sum. On the other hand, the quantity 
$$\sum_{d_j \, | h} a_j, \qquad h \geq 1$$ 
satisfies 
$$\sum_{d_j \, | h} a_j = \sum_{d_j \, | h + m} a_j.$$ Thus, this quantity is
periodic with period $m$. So, by absolute convergence, we may rearrange the sum
and write
\begin{equation}\label{ajzeta}
    \zeta(s)\sum_{j=1}^{d(m)} \frac{a_j}{(d_j)^s}=\sum_{n=0}^{\infty}
    \sum_{k=1}^{m} \frac{b_k}{(mn+k)^s},  
\end{equation}
where
\begin{equation}\label{bkaj}
    b_k = \sum_{d_j \, | k} a_j, \qquad\qquad k=1,\ldots,m.
\end{equation}
Therefore, each $b_k$ is the sum of the $a_j$ with the property that the $j$-th
divisor of $m$ divides $k$. Hence, in terms of the $a_j$, 
the conditions \eqref{bkequation} read
\begin{equation}\label{Aij_entries}
    \sum_{k=1}^m \sum_{d_j\, | k} a_j k^r =0,\qquad \qquad r=0,\ldots,d(m)-1.
\end{equation}

For example, when $m=6$, the coefficient $b_1$ of the $(6n+1)^{-s}$ term is
equal to $a_1$, because only  $d_1=1$ divides $k = 1$ and so only $a_1$
contributes to $b_1$. In another case, the
coefficient $b_3$ of the $(6n+3)^{-s}$ term is $a_1 + a_3$, as only $d_1$ and
$d_3$ (which equal $1$ and $3$, respectively) divide $k = 3$.
Put together, we can easily compute that
$$b_1 = a_1,\quad b_2 = a_1 + a_2,\quad b_3 = a_1 + a_3,\quad b_4 = a_1+a_2.$$
Substituting these back into the condition \eqref{bkequation}, we get 
$$\frac{a_1}{1^s} + \frac{a_1 + a_2}{2^s} + \frac{a_1 + a_3}{3^s} + \frac{a_1 + a_2}{4^s} + \frac{a_1}{5^s} + \frac{a_1 + a_2 + a_3 + a_4}{6^s} = 0.$$
Thus, rearranging, we get 
$$a_1\left(\frac{1}{1^s} + \cdots + \frac{1}{6^s}\right) + a_2\left(\frac{1}{2^s} +\frac{1}{4^s} +\frac{1}{6^s}\right) + a_3\left(\frac{1}{3^s} +\frac{1}{6^s}\right) + a_4\left(\frac{1}{6^s}\right) = 0.$$ 
We want the last equation
to hold for each $s = 0, -1, -2, -3$. This results in a linear system
represented by a $4$-dimensional square matrix $A$ and we want to solve the matrix equation 
$A\boldsymbol{a} = \boldsymbol{0}$, where  $\boldsymbol{a}=[a_1,a_2,a_3,a_4]^T$.
We will later show that 
if $d(m)\ge 4$ then $A$ is singular, 
and so we will obtain a nonzero solution for $\boldsymbol{a}$, and consequently for the $b_k$'s. 

In summary, we construct a $d(m)$-dimensional square matrix 
$A=(A_{ij})$, where $1\le i,j\le d(m)$. The  entries of $A$ are given by
\begin{equation}\label{Aijformula}
A_{ij} = \sum_{n=1}^{\frac{m}{d_j}} (d_j \cdot n)^{i-1}.
\end{equation}
The formula \eqref{Aijformula} arises from 
the linear contraints imposed in \eqref{Aij_entries}. 
Moreover, in view of the findings in the next subsection, 
 provided $d(m)\ge 4$, we can find a nonzero vector 
$$\boldsymbol{a}=[a_1,\ldots,a_{d(m)}]^T$$ 
such that 
$$A\boldsymbol{a} = \boldsymbol{0}.$$ 
Given such a vector $\boldsymbol{a}$, we can solve for the $b_k$ using
\eqref{bkaj} and hence obtain a formula of the form 
\begin{equation}\label{zetaformula}
    \zeta(s) \cdot \left( \sum_{j=1}^{d(m)} \frac{a_j}{(d_j)^s} \right) =
    Z_m(s).
\end{equation}
Although this formula was derived for $\operatorname{Re}(s)>1$,
it follows from 
Theorem~\ref{maintheorem} that the equality holds by analytic continuation
throughout the half-plane $\operatorname{Re}(s)>2-d(m)$.

\subsection{Singularity of $A$}\label{singularity}
We show that $A$ is singular, 
so the kernel of $A$ is nonzero.
Suppose $m$ is such that $d(m)\ge 4$. Then $A$ has at least $4$
columns and rows. 
We claim that the nonzero vector 
$$
\boldsymbol{c}=
\begin{bmatrix}
0,
m^2,
-3m,
2,
0,
0,
\cdots,
0
\end{bmatrix}
$$
is in the left-kernel of the $A$. This will follow on
 showing that for each $j=1,\ldots,m$,
the dot product of $\boldsymbol{c}$ with the $j$-th column of
$A$, denoted by $A_j$, is zero. 

Using the formula \eqref{Aijformula},
the 2nd, 3rd, and 4th entries of the $j$-th column $A_j$ are given by 
\begin{equation}
    \begin{split}
        A_{2j} &= \sum_{n=1}^{\frac{m}{d_j}} (d_j \cdot n) = d_j\cdot
        \frac{1}{2}\cdot \frac{m}{d_j}\left(\frac{m}{d_j} + 1\right) =
        \frac{m(m+d_j)}{2d_j},\\
        A_{3j} &= \sum_{n=1}^{\frac{m}{d_j}} (d_j \cdot n)^2 =
        \left(d_j\right)^2\cdot \frac{1}{6}\cdot
        \frac{m}{d_j}\left(\frac{m}{d_j}+1\right)\left(\frac{2m}{d_j} + 1\right)
        = \frac{m(m+d_j)(2m+d_j)}{6d_j},\\
        A_{4j} &= \sum_{n=1}^{\frac{m}{d_j}} (d_j \cdot n)^3 = (d_j)^3 \cdot\frac{1}{4}\cdot\left(\frac{m}{d_j}\right)^2\left(\frac{m}{d_j}+1\right)^2 = \frac{m^2(m+d_j)^2}{4d_j}.
    \end{split}
\end{equation}
Thus, taking a dot product $\boldsymbol{c}\cdot A_j$ we get 
\begin{equation}
\begin{split}
\boldsymbol{c}\cdot A_j&=\frac{m^3(m+d_j)}{2d_j} - \frac{m^2(m+d_j)(2m+d_j)}{2d_j} + \frac{m^2(m+d_j)^2}{2d_j}\\
&= \frac{m(m+d_j)}{2d_j}(m^2 - m(2m + d_j) + m(m+d_j)) = 0.
\end{split}
\end{equation}
So the left-kernel of the square matrix $A$ is nonzero, and so $A$ must be singular.

For the cases where $d(m) < 4$, we do find that the generated matrix is nonsingular. The two cases to consider are $m=p$ for $p$ prime, which gives $d(m) = 2$, and $m = p^2$ for $p$ a prime, which gives $d(m) = 3$. In the first case, the matrix $A$ reduces to 
$$
A=
\begin{bmatrix}
p&1\\
\frac{p(p+1)}{2}&p
\end{bmatrix},
$$
which is nonsingular. In the second case, the matrix reduces to 
$$
A=
\begin{bmatrix}
p^2&p&1\\
\frac{p^2(p^2+1)}{2}&p\frac{p(p+1)}{2}&p^2\\
\frac{p^2(p^2 + 1)(2p^2 + 1)}{6}&p^2\frac{p(p+1)(2p+1)}{6}&p^4
\end{bmatrix}.
$$
Taking the determinant, we get 
$\det(A) = 
\frac{p^8}{12} - \frac{p^7}{6} + \frac{p^5}{6} - \frac{p^4}{12}$. This is 
$$\frac{p^4}{12}(p^4 - 2p^3 + 2p - 1) = \frac{p^4}{12}(p^2 - 1)(p^2 -2p + 1).$$ 
So again, the determinant is nonzero. 
Thus, if $d(m)<4$, the generated matrix is nonsingular, and our method is not
applicable. For example, the formula \eqref{eta} falls outside the scope of our
method. But if $d(m)\ge 4$, the generated matrix $A$ 
will be singular and our method works. 

\section{Example with $m=24$}
To clarify each step, we provide an example when $m = 24$. In this case, $d(m) =
8$ and $D = \{1,2,3,4,6,8,12,24\}$. So $d_1=1, d_2=2,\ldots,d_8=24$. 
Using the formula \eqref{Aijformula} for $A_{ij}$ we find that 
our $d(m)$-dimensional, or $8$-dimensional, matrix $A=(A_{ij})$ is given by
$$A = \begin{bmatrix}
    24 & 12 & 8 & \dots & 1\\
    300 & 156 & 108 & \dots & 24 \\
    4900 & 2600 & 1836 & \dots & 576 \\
    \vdots & \vdots & \vdots & \ddots & \vdots \\
    {\sum_{n=1}^{24} (n)^{7}} & {\sum_{n=1}^{12} (2n)^{7}} & {\sum_{n=1}^{8} (3n)^{7}} & \dots & (24n)^7 \\
\end{bmatrix}
$$
Using a computer algebra system to compute the so-called 
row-reduced echelon form of this matrix, see \cite{linalg} for example, the result is 
$$\begin{bmatrix}
    1 & 0 & 0 & 0 & 0 & 3 & 56 & 2761 \\
    0 & 1 & 0 & 0 & 0 & {\frac{91}{4}} & -407 & {\frac{-78085}{4}}\\
    0 & 0 & 1 & 0 & 0 & 47 & 792 & 36685\\
    0 & 0 & 0 & 1 & 0 & {\frac{-67}{2}} & -517 & {\frac{-45793}{2}}\\
    0 & 0 & 0 & 0 & 1 & {\frac{29}{4}} & 77 & {\frac{11891}{4}}\\
    0 & 0 & 0 & 0 & 0 & 0 & 0 & 0 \\
    0 & 0 & 0 & 0 & 0 & 0 & 0 & 0 \\
    0 & 0 & 0 & 0 & 0 & 0 & 0 & 0 \\
\end{bmatrix}$$
Thus, one of the nonzero vectors in the kernel is 
$$
\boldsymbol{c}=
\begin{bmatrix} 
56 & -407 & 792 & -517 & 77 & 0 & -1 & 0 \\
\end{bmatrix}^T
$$
Matching this vector with the divisors $d_j$, in accordance with formula \eqref{mainformula}, we get 
$$\zeta(s) \left(56 -\frac{407}{2^s} + \frac{792}{3^s} - \frac{517}{4^s} + \frac{77}{6^s} - \frac{1}{12^s}\right) = Z_{24}(s)$$
To write down the series representation of $Z_{24}(s)$, we use
 formula \eqref{bkaj} to calcualte the $b_k$. For example, to calculate $b_k$
 when $k = 16$, we add the coefficients $a_j$ corresponding to the divisors
 $d_j\in D$ such that $d_j | 16$. Those divisors are $1,2,4$, and $8$, and the
 corresponding $a_j$ are $56, -407,-517$, and $0$. The resulting coefficient is therefore $b_{16}=56 - 407 - 517 + 0 = -868$. 
Repeating this process for every $b_k$, the end result is 
\begin{equation*}
\footnotesize{
    \begin{split}
Z_{24}(s)&=\sum_{n=0}^{\infty} \left(\frac{56}{(24n+1)^s} - \frac{351}{(24n+2)^s} + \frac{848}{(24n+3)^s} - \frac{868}{(24n+4)^s} + \frac{56}{(24n+5)^s}\right. \\
 &+ \frac{518}{(24n+6)^s} + \frac{56}{(24n+7)^s}
 - \frac{868}{(24n+8)^s} + \frac{848}{(24n+9)^s} - \frac{351}{(24n+10)^s} + \frac{56}{(24n+11)^s} \\
 &+ \frac{0}{(24n+12)^s} + \frac{56}{(24n+13)^s} 
 - \frac{351}{(24n+14)^s} + \frac{848}{(24n+15)^s} - \frac{868}{(24n+16)^s} + \frac{56}{(24n+17)^s}\\
 &+ \frac{518}{(24n+18)^s} + \frac{56}{(24n+19)^s}
 - \frac{868}{(24n+20)^s} + \frac{848}{(24n+21)^s} - \frac{351}{(24n+22)^s} + \frac{56}{(24n+23)^s} \\
 &\left.+ \frac{0}{(24n+24)^s}\right).
   \end{split}
}
\end{equation*}
\newline
Combining, we get 
$$\zeta(s) =  \frac{Z_{24}(s)}{56 -\frac{407}{2^s} + \frac{792}{3^s} - \frac{517}{4^s} + \frac{77}{6^s} - \frac{1}{12^s}},$$ 
whenever the denominator on the right-side is nonzero, and 
this converges for $\operatorname{Re}(s) > 2-d(m) = -6$. 

\section{Numerical experiments with $m=6,24,60$}\label{numerical_experiments}
Wolfram Mathematica was used to verify the accuracy of the generated Dirichlet Series and to check the rate of convergence. To do so, precise computed values of $\zeta(\frac{1}{2} + it)$ for $t = 10^4, 10^5, 10^6, 10^7$ were compared to the result of our method with $m=60$, $m = 24$, and $m=6$, as well as to the well-known continuation through the Dirichlet Eta function given in \eqref{eta}. 

We find the series representation corresponding to $m=60$ is given by
$ Z_{60}(s)=\sum_{n=0}^{\infty} (\frac{61768}{(60n+1)^s} - \frac{506228}{(60n+2)^s} + \frac{1657604}{(60n+3)^s} - \frac{2557849}{(60n+4)^s} + \frac{1354748}{(60n+5)^s} + \frac{754819}{(60n+6)^s} + \frac{61768}{(60n+7)^s} - \frac{2557849}{(60n+8)^s} + \frac{1657604}{(60n+9)^s} + \frac{791167}{(60n+10)^s} + \frac{61768}{(60n+11)^s} - \frac{1297395}{(60n+12)^s} + \frac{61768}{(60n+13)^s} - \frac{506228}{(60n+14)^s} + \frac{2950584}{(60n+15)^s} - \frac{2557849}{(60n+16)^s} + \frac{61768}{(60n+17)^s} + \frac{754819}{(60n+18)^s} + \frac{61768}{(60n+19)^s} - \frac{1260454}{(60n+20)^s} + \frac{1657604}{(60n+21)^s} - \frac{506228}{(60n+22)^s} + \frac{61768}{(60n+23)^s} - \frac{1297395}{(60n+24)^s}+
\frac{1354748}{(60n+25)^s} - \frac{506228}{(60n+26)^s} + \frac{1657604}{(60n+27)^s} - \frac{2557849}{(60n+28)^s} + \frac{61768}{(60n+29)^s} + \frac{2052214}{(60n+30)^s} + \frac{61768}{(60n+31)^s} - \frac{2557849}{(60n+32)^s} + \frac{1657604}{(60n+33)^s} - \frac{506228}{(60n+34)^s} + \frac{1354748}{(60n+35)^s} - \frac{1297395}{(60n+36)^s} + \frac{61768}{(60n+37)^s} - \frac{506228}{(60n+38)^s} + \frac{1657604}{(60n+39)^s} - \frac{1260454}{(60n+40)^s} + \frac{61768}{(60n+41)^s} + \frac{754819}{(60n+42)^s}+
\frac{61768}{(60n+43)^s} - \frac{2557849}{(60n+44)^s} + \frac{2950584}{(60n+45)^s} - \frac{506228}{(60n+46)^s} + \frac{61768}{(60n+47)^s} - \frac{1297395}{(60n+48)^s} + \frac{61768}{(60n+49)^s} + \frac{791167}{(60n+50)^s} + \frac{1657604}{(60n+51)^s} - \frac{2557849}{(60n+52)^s} + \frac{61768}{(60n+53)^s} + \frac{754819}{(60n+54)^s} + \frac{1354748}{(60n+55)^s} - \frac{2557849}{(60n+56)^s} + \frac{1657604}{(60n+57)^s} - \frac{506228}{(60n+58)^s} + \frac{61768}{(60n+59)^s} + \frac{0}{(60n+60)^s})$.
Therefore,
$$
\zeta(s) = \frac{Z_{60}(s)}{61768 - \frac{567996}{2^{s}} + \frac{1595836}{3^{s}} - \frac{2051621}{4^{s}} + 
 \frac{1292980}{5^{s}} - \frac{334789}{6^{s}} + \frac{4415}{10^{s}} - \frac{593}{12^{s}}}.
$$
The series representation obtained when $m=24$ was derived in the section prior, and is given by $$\frac{Z_{24}(s)}{56 -\frac{407}{2^s} + \frac{792}{3^s} - \frac{517}{4^s} + \frac{77}{6^s} - \frac{1}{12^s}}.$$ 
The series representation when $m=6$ is $$\frac{\sum_{n=0}^{\infty} \left(\frac{1}{(6n+1)^s}-\frac{4}{(6n+2)^s}+\frac{6}{(6n+3)^s}-\frac{4}{(6n+4)^s}+\frac{1}{(6n+5)^s}\right)}{1 -\frac{5}{2^s} + \frac{5}{3^s} - \frac{1}{6^s}}.$$
And the series representation from \eqref{eta} is 
$$\frac{\sum_{n=0}^{\infty}
\left(\frac{1}{(2n+1)^s}-\frac{1}{(2n+2)^s}\right)}{1 -\frac{2}{2^s}}.$$

In order to test the convergence rate of these series, we computed the summation
over $n$ in formula \eqref{mainformula} 
up to varying numbers of terms $N$. The choices of $N$ that we made 
were 
the minimum necessary to be within a prescribed desired accuracy of less than
$0.001$,$0.0001$ and $0.00001$. The minimum $N$ that achieved this 
is given in the displayed tables. 
Note, however, that our criterion for choosing $N$ could be
occasionally inconsistent. For example, for some $t$ there could be
$N$ that by chance brings the sum to within the prescribed accuracy.
This appears to be the case for $m=2$ when $t$ is small.
Nevertheless, one can still glean distinct patterns despite the occasional
apparent inconsistency.

\begin{table}[ht]
\centering
\footnotesize{
\begin{tabular}{ ||p{1.6cm}|p{1.6cm}||p{1.6cm}||p{1.6cm}||p{1.6cm}| }
 \hline
  $s=1/2+it$  & \multicolumn{4}{|c|}{Minimum $N$ necessary for error of magnitude $<0.001$} \\
 \hline
 $t$& $m=60$& $m=24$& $m=6$& $m=2$\\
 \hline
 $10^4$ &$2.4 \times 10^2$  & $4.5 \times 10^2$ & $2 \times 10^3$ & $7.5 \times 10^4$\\
 $10^5$&$1.7 \times 10^3$& $4.25 \times 10^3$ & $1.8 \times 10^4$ & $6.5 \times 10^4$\\
 $10^6$&$2.7 \times 10^4$ & $4.3 \times 10^4$ & $1.7 \times 10^5$ & $2.4 \times 10^5$\\
 $10^7$ &$2.7 \times 10^5$& $4.1 \times 10^5$ & $1.7 \times 10^6$ & $2.4 \times 10^6$\\
 \hline
\end{tabular}

    \vspace{3mm}

\begin{tabular}{ ||p{1.6cm}|p{1.6cm}||p{1.6cm}||p{1.6cm}||p{1.6cm}| }
 \hline
  $s=1/2+it$  & \multicolumn{4}{|c|}{Minimum $N$ necessary for error of
    magnitude $<0.0001$} \\
 \hline
 $t$& $m=60$& $m=24$& $m=6$& $m=2$\\
 \hline
 $10^4$&$3.2 \times 10^2$   & $8.1 \times 10^2$ & $2.9 \times 10^3$ & $7 \times 10^6$\\
 $10^5$&$2.7 \times 10^3$& $8 \times 10^3$ & $2.3 \times 10^4$ & $5.1 \times 10^6$\\
 $10^6$&$3.2 \times 10^4$ & $8 \times 10^4$ & $2.1 \times 10^5$ & $6 \times 10^6$\\
 $10^7$ &$3.2 \times 10^5$& $8 \times 10^5$ & $1.8 \times 10^6$ & $7.3 \times 10^7$\\
 \hline
\end{tabular}
 
    \vspace{3mm}

    \begin{tabular}{ ||p{1.6cm}|p{1.6cm}||p{1.6cm}||p{1.6cm}||p{1.6cm}| }
 \hline
  $s=1/2+it$  & \multicolumn{4}{|c|}{Minimum $N$ necessary for error of
     magnitude $<0.00001$} \\
 \hline
 $t$& $m=60$& $m=24$& $m=6$& $m=2$\\
    \hline
 $10^4$&$3.3 \times 10^2$   &  $8.8 \times 10^3$ & $5 \times 10^3$ & $7 \times 10^8$\\
 $10^5$&$3.2 \times 10^3$& $8.3 \times 10^3$ & $3.7 \times 10^4$ & $5.1 \times 10^8$\\
 $10^6$&$3.2 \times 10^4$ & $8.2 \times 10^4$ & $3.1 \times 10^5$ & $6 \times 10^8$\\
 $10^7$&$3.2 \times 10^5$ & $8 \times 10^5$& $2.4 \times 10^6$ & $7.3 \times 10^9$\\
 \hline
\end{tabular}
}
    \vspace{2mm}
    \caption{\small{Minimum number of terms needed for various $m$ and $s$ to
    be within an error of magnitude under $0.001, 0.0001$, and $0.00001$.}}
\end{table}

The $m=2$ case, the Dirichlet Eta Function, is known to converge on the critical
line with error term
of order $\frac{1}{\sqrt{N}}$; see \cite{hiary} for example. This is reflected
in the tables, as decreasing the error by a factor of $10$ (from $.001$ to
$.0001$, or from $.0001$ to $.00001$) for the same value of $t$ requires $100$ times as many terms. 
In comparison, for each of the $m=6$, $m=24$, and $m=60$ cases, the minimum $N$
needed does not increase nearly as fast as the prescribed accuracy is decreased.
In almost all cases, for a given prescribed accuracy, 
the number of terms needed scales approximately linearly with
$t$ (or with the magnitude of $s$). 

\begin{remark}
The number $N$ refers to the number of inner sums being added in the right-side of \eqref{mainformula}, and so is the upper limit of the summation over $n$. To get the total number of individual terms added one should multiply by $m$ (since each inner sum has $m$ terms).
\end{remark}

\subsection{Error analysis}
In this section, we will approximate the error resulting from truncating our
formula for $Z_m(s)$ at $n=N$.
Consider
\begin{equation}\label{firstterm}
\frac{f_{d(m)}(s)}{(mn + m/2)^{s+d(m)}}\sum_{k=1}^{m} b_k k^{d(m)}.
\end{equation}
As pointed out in the remark following Theorem~\ref{maintheorem},
this is the first nonzero term in the Taylor expansion used in the proof of the
theorem.
For  $\operatorname{Re}(s)>2-d(m)$, we use monotonicity to estimate 
\begin{equation}
    \begin{split}
        \int_{n=N}^{\infty} \frac{1}{(mn+m/2)^{\operatorname{Re}(s)+d(m)}} \mathop{dn} &<\sum_{n=N}^{\infty}
\frac{1}{(mn+m/2)^{\operatorname{Re}(s)+d(m)}} \\
        & <\int_{n=N-1}^{\infty} \frac{1}{(mn+m/2)^{\operatorname{Re}(s)+d(m)}} \mathop{dn}.
    \end{split}
\end{equation}
Therefore,
\begin{equation}
    \begin{split}
        \sum_{n=N}^{\infty} \frac{1}{(mn+m/2)^{\operatorname{Re}(s)+d(m)}} &<
 \frac{1}{m(\operatorname{Re}(s)+d(m)-1)(mN-\frac{m}{2})^{\operatorname{Re}(s)+d(m)-1}}
        \\
        &=: \mathcal{T}(s,N).
  \end{split}
\end{equation}
If the behavior of the error 
is mainly determined by \eqref{firstterm}, as expected, 
then 
the truncation error resulting from using $n=N$ terms in the formula for
$Z_m(s)$ -- that is the difference between the actual value of $\zeta(s)$ and
our truncated formula -- is approximately 
\begin{equation}\label{error1}
    \left|
    \mathcal{T}(s,N)\cdot \frac{\prod_{u=0}^{d(m)-1}(s+u)}{d(m)!} 
    \cdot \frac{\sum_{k=1}^{m} b_k k^{d(m)}}{\sum_{j=1}^{d(m)} a_j (d_j)^{-s}} \right|.
\end{equation}

We simplify the estimate \eqref{error1} some more by specializing to the
parameters used in our 
set of numerical experiments, 
which we conducted on the critical line and included
heights $t$ not too small. 
Since $|s|$ is significantly larger than $m$ in our experiments, it is reasonable to
approximate $|\prod_{u=0}^{d(m)-1}{(s+u)}| \approx |s|^{d(m)}$. Also, since we
are working on the critical line, we approximate
$|\mathcal{T}(s,N)|\approx (m(mN)^{d(m)-1/2})^{-1}$.
So, on the critical line, the estimate \eqref{error1} behaves like 
\begin{equation}\label{error3}
    \left|\frac{1}{m(mN)^{d(m)-1/2}} \cdot 
    \frac{|s|^{d(m)}}{d(m)!} 
\cdot \frac{\sum_{k=1}^{m} b_k k^{d(m)}}{\sum_{j=1}^{d(m)} a_j (d_j)^{-s}}
    \right|
\end{equation}

Using \eqref{error3}, 
we see why in our set of experiments with $m=6,24,60$, increasing $t$ tenfold
requires approximately tenfold increase in $N$ to maintain the same level of
accuracy. For increasing $t$ in such a way multiplies $|s|$ by approximately $10$,
which multiplies the numerator in \eqref{error3} by approximately $10^{d(m)}$.
On the other hand, increasing $N$ tenfold multiplies the
denominator by a factor of $10^{d(m)-1/2}$. 
For the same accuracy, then, it suffices to multiply $N$  
by about $10$. This behavior is clearly reflected in
 the tables from the prior subsection.

We can also interpret the rate of convergence observed in our experiments 
using the estimate \eqref{error3}. Generally, to improve the accuracy by a
multiplicative factor $1/\eta$, we need to multiply $N$ by a factor $\kappa$
such that $\kappa^{d(m)-1/2} = \eta$, so $\kappa =
\eta^\frac{1}{d(m)-1/2}$. 
Table~\ref{error_table_2} demonstrates 
that this expected rate of convergence is at work.
The entries in the table show this expected trend. In our experiments we have
$s=1/2+it$. For example, for $\eta=10$, we consider the corresponding $\kappa$'s
for various $m$. When $m=6$, we get $\kappa = 10^{2/7} \approx 1.93$. When
$m=24$, we get $\kappa = 10^{2/15} \approx 1.36$. And when $m=60$, we get
$\kappa = 10^{2/23} \approx 1.22$. We see that $\kappa$'s we find empirically
match these expected values fairly well. When $m=6$, we get $4.70/2.43 \approx 1.93$ in one case and  $2.43/1.27 \approx 1.91$ in another case. Similarly, when $m = 24$, we get $1.96/1.46 \approx 1.34$ in one case, and $1.46/1.14 \approx 1.28$ in another case. And when $m = 60$ we get $4.74/4.08 \approx 1.16$ and $4.08/3.61 \approx 1.13$.

\begin{table}[ht]
\centering
\footnotesize{
\begin{tabular}{ ||p{2.2cm}||p{2.2cm}||p{2.2cm}||p{2.2cm}|}
 \hline
    & \multicolumn{3}{|c|}{Minimum $N$ necessary to achieve a specified accuracy} \\
 \hline
 Accuracy & $m=60$ & $m=24$ & $m=6$\\
 \hline
 $10^{-6}$&$3.33 \times 10^3$&  $9.4 \times 10^3$ & $6.7 \times 10^4$\\
 $10^{-7}$&$3.61 \times 10^3$& $1.14 \times 10^4$ & $1.27 \times 10^5$\\
 $10^{-8}$&$4.08 \times 10^3$&$1.46 \times 10^4$ & $2.43 \times 10^5$\\
 $10^{-9}$&$4.74 \times 10^3$ & $1.96 \times 10^4$& $4.70 \times 10^5$\\
 \hline
\end{tabular}
}
\vspace{2mm}
    \caption{\small{
    Minimum $N$ to achieve a prescribed accuracy for $s =
    \frac{1}{2} + 10^5i$. The case $m=2$ was excluded 
    as its rate of convergence is known.}}
\label{error_table_2}
\end{table}

When compared to the Dirichlet Eta Function ($m=2$) case, representations with
higher $m$ converge much faster: for $m=2$, increasing precision by a factor of
$10$ requires increasing the number of terms by a factor of $100$, while for the
$m=60$ case, once $N$ is large enough, it appears to only require increasing the number of terms by a factor of about $1.22$ asymptotically. As a result, these series offer more efficient and fairly simple ways to compute the zeta function.

\printbibliography

\end{document}